\documentclass[letterpaper, 10 pt, journal, twoside]{ieeetran}  

\IEEEoverridecommandlockouts                              

\usepackage[english]{babel}

\usepackage{amsmath,amssymb,amsthm} 

\newtheoremstyle{theoremdd}
{\topsep}
{\topsep}
{\itshape}
{0pt}
{\bfseries}
{:}
{ }
{\thmname{#1}\thmnumber{ #2}\thmnote{ (#3)}}

\theoremstyle{theoremdd}

\newtheorem{theorem}{Theorem}[section]

\newtheorem{lemma}[theorem]{Lemma}

\usepackage{chngcntr}
\usepackage{apptools}
\AtAppendix{\counterwithin{lemma}{section}}

\usepackage{mathrsfs}
\usepackage{graphicx} 
\usepackage{epsfig} 
\usepackage{epstopdf}
\usepackage{color}
\usepackage{caption}
\usepackage{float}
\usepackage{mathtools}

\usepackage{cite}
\usepackage[ruled,vlined]{algorithm2e}

\usepackage{tikz}
\usetikzlibrary{shapes,arrows,positioning,calc,chains}

\makeatletter
\DeclareRobustCommand{\rvdots}{%
	\vbox{
		\baselineskip4\p@\lineskiplimit\z@
		\kern-\p@
		\hbox{.}\hbox{.}\hbox{.}
}}
\makeatother

\usepackage{lipsum}
\makeatletter
\def\footnoterule{\relax%
	\kern-5pt
	\hbox to \columnwidth{\hfill\vrule width .9\columnwidth height 0.4pt\hfill}
	\kern4.6pt}
\makeatother

\makeatletter
\newcommand{\figcaption}{\def\@captype{figure}\caption}
\newcommand{\tabcaption}{\def\@captype{table}\caption}

\usepackage{color,hyperref}
\definecolor{darkblue}{rgb}{0.0,0.0,0.3}
\hypersetup{colorlinks,breaklinks,linkcolor=darkblue,urlcolor=darkblue,anchorcolor=darkblue,citecolor=darkblue}

\makeatother
\title{\LARGE \bf Necessary and Sufficient Conditions for \\Frequency-Based Kelly Optimal Portfolio
}

\author{\large Chung-Han Hsieh,$^{*}$ \textit{Member, IEEE} 
	\thanks{\hskip -10pt ${}^*$Chung-Han Hsieh is with the Department of Electrical and Computer Engineering, University of Wisconsin, Madison, WI 53706, USA. E-mail: \href{mailto: chunghan.hsieh@wisc.edu}{chunghan.hsieh@wisc.edu}. Parts of this paper is based on the preliminary work given in \cite{Hsieh_Barmish_Gubner_2018_ACC} and \cite{Hsieh_Gubner_Barmish_2018_CDC}.} 
}

\begin{document}

	\maketitle
	\thispagestyle{empty}
	\pagestyle{empty}
	
	
	\begin{abstract} 
		In this paper, we consider a discrete-time portfolio with $m \geq 2$ assets optimization problem  which includes the \textit{rebalancing~frequency} as an additional parameter in the maximization. 
		The so-called \textit{Kelly Criterion} is used as the performance metric; i.e., maximizing the expected logarithmic growth of a trader's account, and the portfolio obtained is called the frequency-based Kelly optimal portfolio. 
		The focal point of this paper is to extend upon the results of our previous work to obtain various optimality characterizations on the portfolio. 
		To be more specific,
		using Kelly's criterion in our frequency-based formulation, we first prove  necessary and sufficient conditions for the frequency-based Kelly optimal  portfolio.
		With the aid of these conditions, we then show several new optimality characterizations such as \textit{expected ratio optimality} and \textit{asymptotic relative optimality}, and a result which we call the Extended Dominant Asset Theorem. 
		That is, we prove that the $i$th asset is \textit{dominant} in the portfolio if and only if the Kelly optimal portfolio consists of that asset only. 
		The word ``extended" on the theorem comes from the fact that it was only a sufficiency result that was proved in our previous work. Hence,  in this paper, we improve it to involve a proof of the necessity part.  
		In addition, the trader's survivability issue (no bankruptcy consideration) is also studied in detail in our frequency-based trading framework. 
		 Finally, to bridge the theory and practice,  we propose a simple trading algorithm  using the notion called \textit{dominant asset condition} to decide when should one triggers a trade. The corresponding trading performance using historical price data is reported as supporting evidence. 
	\end{abstract}
	
	\vspace{3mm}
	
	
	%
	\section{Introduction}
	\label{SECTION: INTRODUCTION}
		The takeoff point for this paper is the classical Kelly trading problem~\cite{Kelly_1956, Cover_Thomas_2012, Thorp_2006, Rotando_Thorp_1992,Algoet_Cover_1988}, which calls for maximizing the Expected Logarithmic Growth~(ELG) of a trader's account. 
		To be more specific, the problem is often formulated by a sequence of trades with independent and identically distributed (i.i.d.) returns with known probability distribution. 
		The trader's objective is  to specify a fraction~$K$ of its account value at each stage seeking to maximize the ELG at the terminal stage.
		  While many of the existing papers contributed on the Kelly's problem and its application to stock trading; e.g., see~\mbox{\cite{Cover_Thomas_2012,Lo_Orr_Zhang,Luenberger_2011,Algoet_Cover_1988,Maclean_Thorp_Ziemba_2010,Thorp_2006,Rotando_Thorp_1992}}, the effects of \textit{rebalancing frequency} is still \textit{not} heavily considered into the existing~literature. 
	
%
Some initial results along these lines regarding rebalancing frequency effects can be found in~\cite{Kuhn_Luenberger_2010, Das_Kaznachey_Goyal_2014, Das_Kaznachey_Goyal_2015} and our most recent work in \cite{Hsieh_Barmish_Gubner_2018_ACC,Hsieh_Gubner_Barmish_2018_CDC, Hsieh_Dissertation}. 
Indeed, in \cite{Kuhn_Luenberger_2010}, a portfolio optimization with returns following a continuous geometric Brownian motion was considered. However, only two extreme cases: High-frequency trading and buy and hold were emphasized in their results.
On the other hand, in~\cite{Das_Kaznachey_Goyal_2014} and \cite{Das_Kaznachey_Goyal_2015},  a portfolio optimization was considered with the constant gain~$K$ selected without regard for the frequency with which the portfolio rebalancing is done. 
Subsequently, when this same gain~$K$ is used to find an optimal rebalancing period, the resulting levels of ELG are arguably suboptimal. 

In contrast to~\cite{Kuhn_Luenberger_2010} and \cite{Das_Kaznachey_Goyal_2014},  our formulation to follow, achieved by adopting our previous work published in~\cite{Hsieh_Barmish_Gubner_2018_ACC} and \cite{Hsieh_Gubner_Barmish_2018_CDC}, considers full range of rebalancing frequencies and both the probability distribution of the returns and the time interval between rebalances are arbitrary.  
That is, we deal with what we view to be a more appropriate  \textit{frequency-based Kelly trading formulation} and seek an optimal portfolio which depends on the rebalancing frequency.

\subsection{Idea of Frequency-Based Formulation}
Specifically, within this frequency-based trading context, we let~$\Delta t$ be the time between trade updates and~\mbox{$n \geq 1$} be the number of steps between rebalancings. 
Then the frequency is~\mbox{$f:=1/ (n \Delta t)$}.
In the sequel, we may call the quantity $n$ to be the \textit{rebalancing period}.
Now,
letting~$V(k)$ denote the trader's account value at stage $k$, the trader invests~$KV(0)$ with~\mbox{$K \geq 0$} at stage~$k = 0$ and waits~$n \geq 1$ steps before updating the trade~size. 
After each trade, the broker takes its share and the balance of the money is left to ``ride'' with resulting profits or losses  viewed as ``unrealized'' until stage $n$ is  reached.
When~$n$ is small, this is viewed as the high-frequency case, and when~$n$ is large, one use the term ``buy~and~hold". 

\subsection{Plan for the Remainder of the Paper}
In Section~\ref{SECTION: Problem Formulation}, we first recall our frequency-based formulation considered in \cite{Hsieh_Barmish_Gubner_2018_ACC} and \cite{Hsieh_Gubner_Barmish_2018_CDC}. 
Then, in Section~\ref{SECTION: Main Results}, based on the formulation, we offer our main result which gives necessary and sufficient conditions for the frequency-based optimal Kelly portfolio.
 In addition, several technical results regarding the various optimality conditions are also provided; e.g., extended dominant asset theorem, the expected ratio optimality, and asymptotic relative optimality are proved. 
In Section~\ref{SECTION: Dominant Ratio Trading Algorithm}, we propose a simple trading algorithm which uses the idea of extended dominant asset theorem to determine when should one trigger a trade on an underlying asset or not. 
Several back-testing simulations using historical prices are provided to support the trading performance of the algorithm. 
In Section~\ref{SECTION: Conclusion}, a concluding remark is provided. Finally, in Appendix, we also address an important issue regarding survivability~\mbox{(no-bankruptcy)}.

\section{Problem Formulation}
\label{SECTION: Problem Formulation}
To study the effect of rebalancing frequency in portfolio optimization problems, as seen in Section~\ref{SECTION: INTRODUCTION}, let $n \geq 1$ being the number of steps between rebalancings. 
 For $k=0,1,\dots,$ we consider a trader who is forming a portfolio consisting of~$m \geq 2$ assets and assume that at least one of them is riskless with nonnegative rate of return $r \geq 0.$ That is, if an asset is riskless, its return is deterministic and is treated as a degenerate random variable with value $r$ for all $k$ with probability one. 
Alternatively, if Asset~$i$ is a stock whose price at time~$k$ is~$S_i(k)>0$, then its return is
\[
X_i(k) = \frac{S_i(k+1) - S_i(k)}{S_i(k)}.
\]
In the sequel, for stocks, we assume that the return vectors~$
X(k):=\left[X_1(k) \, X_2(k)\,  \cdots \,X_m(k)\right]^T
$
have a known distribution and have components $X_i(\cdot)$ which can be arbitrarily correlated.\footnote{Again, if the $i$th asset is riskless, then we put $X_i(k) = r \geq 0$ with probability one. If a trader maintains \textit{cash} in its portfolio, then this corresponds to the case~\mbox{$r=0.$}} 
We also assume that these vectors are i.i.d. with components satisfying
$
X_{\min,i}  \leq X_i(k) \leq X_{\max,i}
$
with known bounds above and with $X_{\max,i}$ being finite and~\mbox{$X_{\min,i} > -1$}.
The latter constraint on $X_{\min,i}$ means that the loss per time step is limited to less than~$100\%$ and the price of a stock cannot drop to zero.

\subsection{Feedback Control Perspectives}
Consistent with the literature~\cite{Barmish_Primbs_TAC_2016,Hsieh_Barmish_Gubner_2019_TAC,Hsieh_Barmish_Gubner_2018_ACC,Hsieh_Gubner_Barmish_2018_CDC, Zhang_2001, Primbs_ACC_2007, Hsieh_Dissertation}, we bring the control-theoretic point of view into our problem formulation.
 That is, the system output at stage~$k$ is taken to be the trader's account value~$V(k)$ and the $i$th feedback gain 
$
0 \leq K_i \leq 1
$
represents the fraction of the account allocated to the~\mbox{$i$th} asset for~$i=1,\dots,m$. 
Said another way,  the~$i$th controller is  a linear feedback of the form
$
I_i(k) = K_iV(k).
$
Since~$K_i \geq 0$, the trader is going {\it long}.\footnote{In finance, a \textit{long} trade means that the trader purchases shares from the broker in the hope of making a profit from a subsequent rise in the price of the underlying stock.} 
In view of the above and recalling that there is at least one riskless asset available, without loss of generality,
we consider the unit simplex constraint
$$
K \in {\mathcal K} := \left\{K \in \mathbb{R}^{m}: K_i \geq 0 \text{ for all $i$}, \; \sum_{i=1}^m K_i = 1 \right\}
$$ 
 which is classical in finance; e.g., see \cite{Luenberger_2011, Cover_Thomas_2012, Hsieh_Gubner_Barmish_2018_CDC}.
That is, with~\mbox{$K \in \mathcal K$}, we have a guarantee that~100\% of the account is invested. Moreover, we claim that the constraint set $\mathcal{K}$ assures trader's survivability; i.e., no bankruptcy is assured; see Appendix for a proof of this important property. 

\subsection[Frequency-Dependent Dynamics and Feedback Configuration]{Frequency-Dependent Dynamics and Feedback Setting} 
%
%
%
Letting $n \geq 1$ be the number of steps between rebalancings,  at time~\mbox{$k=0$}, the trader begins with initial investment control 
$$
u(0) = \sum_{i=1}^m K_i V(0)
$$
and waits~$n$ steps in the spirit of buy and hold. Then, when~\mbox{$k=n$}, the investment control is updated to be~\mbox{$
u(n) = \sum_{i=1}^m K_i V(n).
$}
Now, to study the performance which is dependent on rebalancing frequency, for \mbox{$i=1,2,\ldots,m$}, we use the  \textit{compound~returns}
\[
\mathcal{X}_{n,i} := \prod_{k=0}^{n-1} (1+X_i(k)) -1
\]
which are readily seen to satisfy~\mbox{$\mathcal{X}_{n,i} > -1$} for all $n \geq 1$ and we work with the random vector~$\mathcal{X}_n$ having $i$th
component~$\mathcal{X}_{n,i}$.
Then, for an initial account value~$V(0)>0$ and rebalancing period $n \geq 1$, the corresponding account value at stage $n$ is described by the stochastic recursion
$$
V(n) = (1 + K^T \mathcal{X}_n )V(0). 
$$
In the sequel, we may sometimes write $V(n,K)$ to emphasize the dependence on the feedback gain $K$.

\subsection[Frequency-Dependent Optimization Problem]{Frequency-Dependent Optimization Problem} 
Consistent with our prior work in \cite{Hsieh_Barmish_Gubner_2018_ACC} and \cite{Hsieh_Gubner_Barmish_2018_CDC}, for any rebalancing period~$n \geq 1$, we 
study the problem of maximizing   the expected logarithmic~growth 
\begin{align*}
{g_n}(K) 
&:=  \frac{1}{n}\mathbb{E}\left[ \log \frac{V(n,K)}{V(0)} \right] \\[1ex]
&= \frac{1}{n}\mathbb{E}\left[ {\log (1 + {K^T}{\mathcal{X}_n})} \right]
\end{align*}
 and we use $g_n^*$ to denote the associated optimal expected logarithmic~growth.
It is readily verified that $g_n(K)$ is concave in~$K$.
Furthermore, any vector~\mbox{$K^* \in \mathcal{K} \subset \mathbb{R}^m$} satisfying~$g_n(K^*) = g_n^*$ is called a \textit{Kelly optimal feedback gain}. 
The portfolio which uses the Kelly optimal feedback gain is called \textit{frequency-based Kelly optimal portfolio}.

%

	\section{Results On Optimality}
	\label{SECTION: Main Results}
	\vspace{0mm}
In this section, we provide necessary and sufficient conditions which characterize the frequency-based Kelly optimal~portfolio.

	\begin{theorem}[Necessity and Sufficiency]\label{thm: KKT optimality}
	The feedback gain $K^*$ is optimal to the frequency-dependent optimization problem described in Section~\ref{SECTION: Problem Formulation} if and only if for $i=1,2,\dots,m$,
\begin{align*}
\mathbb{E} \left[\frac{1 +\mathcal{X}_{n,i}}{ 1 + {K^*}^T \mathcal{X}_{n}}\right] = 1, \;\;\; \text{ if } K_i^* >0
\end{align*}
\begin{align*}
\mathbb{E} \left[\frac{1+\mathcal{X}_{n,i}}{1+ {K^*}^T \mathcal{X}_{n}}\right] \leq 1,\;\;\;  \text{ if } K_i^*=0 .
\end{align*}

	\end{theorem}

\vspace{0mm}	
\begin{proof} To prove necessity,
 define  $\mathcal{R}_n := \mathcal{X}_n + \bf{1}$ representing the total return with~$i$th component \mbox{$\mathcal{R}_{n,i} = \mathcal{X}_{n,i}+1 $} and
	\mbox{${\bf 1} := [1 \; 1\; \cdots \; 1]^T \in \mathbb{R}^m$}. 
	We now consider the frequency-dependent optimization problem as an equivalent constrained convex minimization problem as follows:
	\begin{align*}
	&\max_K -\mathbb{E}[\log K^T \mathcal{R}_n]\\
	&\text{subject to}\\ 	
	&\hspace{5mm}   K^T {\bf 1}-1 = 0 ;
	\\
	& \hspace{5mm} -K^T e_i \leq 0, \;\;\;  i=1,2,...,m
	\end{align*}
	where  $e_i $ is unit vector having $1$ at $i$th component. 
	Then the Karush-Kuhn-Tucker Conditions, see e.g.,~\cite{Boyd_Vandenberghe_2004},  tell us that if~$K^*$ is a local maximum then there is a scalar $\lambda \in \mathbb{R}^1$ and a vector~\mbox{$\mu \in \mathbb{R}^m$} with component~$\mu_j \geq 0$ such that 
	\begin{align*}
	\nabla (-\mathbb{E}[\log {K^*}^T \mathcal{R}_n]) + \lambda  {\bf 1} - \sum_{i=1}^m \mu_i e_i = \textbf{0} 
	\end{align*}
	with $\textbf{0} \in \mathbb{R}^m$ being zero vector and 
$
	\mu_j {K^*}^T e_j = 0
$
	for \mbox{$ j = 1,2,\dots,m$}.
	This implies that for $j = 1,\dots,m$, we have~$\mu_j{K_j^*} = 0$ and
	\begin{align} \label{eq:01}
	-\mathbb{E} \left[\frac{\mathcal{R}_{n,j}}{ {K^*}^T \mathcal{R}_{n}}\right] + \lambda- \mu_j = 0. 
	\end{align}
	We note here that the interchanging of differentiation and expectation is justifiable since $\mathcal{X}_{n,i}$ is bounded.
	Now we take a weighted sum of  equation~(\ref{eq:01}); i.e.,
	\begin{align*}
	\sum_{j=1}^m K_j^* \left( -\mathbb{E} \left[\frac{\mathcal{R}_{n,j}}{ {K^*}^T \mathcal{R}_{n}}\right] + \lambda - \mu_j \right) = 0
	\end{align*}
	which leads to
	\begin{align*}
	-\sum_{j=1}^m K_j^*  \mathbb{E} \left[\frac{\mathcal{R}_{n,j}}{ {K^*}^T \mathcal{R}_{n}}\right] + \sum_{j=1}^m K_j^* \lambda   - \sum_{j=1}^m K_j^*\mu_j = 0.
	\end{align*}
Using the facts that $\mu_j K_j^* = 0$ for all $j$ and $\sum_{j=1}^m K_j^* = 1$, we have
	\begin{align}\label{eq:02}
	-\sum_{j=1}^m K_j^*  \mathbb{E} \left[\frac{\mathcal{R}_{n,j}}{ {K^*}^T \mathcal{R}_{n}}\right] +  1 \cdot \lambda   + 0 = 0.
	\end{align}
Note that
	\begin{align*}
	\sum_{j=1}^m K_j^*  \mathbb{E} \left[\frac{\mathcal{R}_{n,j}}{ {K^*}^T \mathcal{R}_{n}}\right] &=   \mathbb{E} \left[\frac{{K^*}^T\mathcal{R}_{n}}{ {K^*}^T \mathcal{R}_{n}}\right] = 1.
	\end{align*}
	Thus, substituting the result above back into equation~(\ref{eq:02}), we obtain $\lambda = 1$. This tells us that for
	$j = 1,\dots,m$,
	\begin{align*}
	-\mathbb{E} \left[\frac{\mathcal{R}_{n,j}}{{K^*}^T \mathcal{R}_{n}}\right] + 1- \mu_j = 0 
	\end{align*}
	and
$
	\mu_j{K_j^*} = 0.
$
	Thus, to sum up, if $K_j^* >0$, implies that $\mu_j =0$ and 
	\begin{align*}
	\mathbb{E} \left[\frac{\mathcal{R}_{n,j}}{{K^*}^T \mathcal{R}_{n}}\right] = 1.
	\end{align*}
	If $K_j^*=0,$ implies that $\mu_j \geq 0$ implies that
	\begin{align*}
	\mathbb{E} \left[\frac{\mathcal{R}_{n,j}}{{K^*}^T \mathcal{R}_{n}}\right] \leq 1.
	\end{align*}
	Now, transforming the $\mathcal{R}_n$ back to $\mathcal{X}_n$ by $\mathcal{R}_n = \mathcal{X}_n + 1$ and using the fact that $\sum_{i=1}^m K_i^* = 1$ again, we obtain the desired conditions.
	Finally, by concavity of $\mathbb{E}[\log K^T\mathcal{R}_n]$, the conditions above are also~sufficient. \qedhere
\end{proof}

\textbf{Remarks:}  It is interesting to note that if $n=1$, then Theorem~\ref{thm: KKT optimality} reduces to the classical result in classical Kelly theory; see \cite[Theorem 16.2.1]{Cover_Thomas_2012}.   
Additionally, Theorem~\ref{thm: KKT optimality} is also closely related to the Dominant Asset Theorem given in our prior work \cite{Hsieh_Gubner_Barmish_2018_CDC}.
For the sake of completeness, we recall the statement of the theorem as follows: Given a collection of $m \geq 2$ assets, if Asset~$j$ is \textit{dominant}; i.e., Asset~$j$ satisfies
\[
\mathbb{E}\left[ \frac{1+X_i(0)}{1+X_j(0)}\right]\leq 1
\] for every other asset $ i \neq j$, then $K^* = e_j.$ Thus, $K_i^* = 0$ for $i \neq j.$\footnote{Intuitively speaking, the Dominant Asset Theorem tells us that when condition is right, one should ``bet the farm."} In fact, this result can be viewed as a special case of Theorem~\ref{thm: KKT optimality}. 
It should be also noted that the Dominant Asset Theorem is about sufficiency on optimal $K^*$--- not necessity. 
Fortunately, with the aids of Theorem~\ref{thm: KKT optimality}, we are now able to prove the missing  part on necessity of Dominant Asset Theorem. This is summarized in the next theorem to follow.

\begin{theorem}[Extended Dominant Asset Theorem]\label{thm: Extended Dominant Asset Theorem}
	The optimal Kelly feedback gain $K^*=e_j$ if and only if $$\mathbb{E}\bigg[\frac{1+X_i(0)}{1+X_j(0)}\bigg]\leq 1.$$
\end{theorem}

\begin{proof}
	The sufficiency is proved in our prior work in~\mbox{\cite[Dominant Asset Theorem]{Hsieh_Gubner_Barmish_2018_CDC}}. Hence, for the sake of brevity, we only provide a proof of necessity here.  Assuming that $K^*=e_j$, we must show the desired inequality holds. Applying Theorem~\ref{thm: KKT optimality}, it follows that 
	for $i \neq j,$ $K_i^*=0$ and
	\[
	\mathbb{E} \left[\frac{1+\mathcal{X}_{n,i}}{1+ {K^*}^T\mathcal{X}_{n,j}}\right] = \mathbb{E} \left[\frac{1+\mathcal{X}_{n,i}}{1+ \mathcal{X}_{n,j}}\right] \leq 1.
	\]
	Using the definition of $\mathcal{X}_{n,i} = \prod_{k=0}^{n-1}(1+X_i(k)))-1$, the equality above indeed implies that
	\begin{align*}
	\mathbb{E} \left[ \prod_{k=0}^{n-1}\frac{1+X_i(k)}{1+ X_j(k)}\right] \leq 1.
	\end{align*}
	Since $X_i(k)$ are i.i.d., in $k$, we have
	\begin{align}\label{eq:03}
	\left( \mathbb{E} \left[\frac{1+X_i(0)}{1+ X_j(0)}\right] \right)^n \leq 1.
	\end{align}
	Note that $X_i(0)>-1$ for all $i=1,2,\dots,m$, it follows that the ratio 
	$$\frac{1+X_i(0)}{1+ X_j(0)} > 0$$ with probability one; hence its expected value is also strictly positive. Thus, in combination with inequality~(\ref{eq:03}), we conclude
	\[
	\mathbb{E} \left[\frac{1+X_i(0)}{1+ X_j(0)}\right] \leq 1.  \qedhere
	\]
\end{proof}

\textbf{Remark:} When the condition
	\begin{align}\label{eq: dominant ratio}
		\mathbb{E}\bigg[\frac{1+X_i(0)}{1+X_j(0)}\bigg]\leq 1
	\end{align}
	the Extended Dominant Asset Theorem~\ref{thm: Extended Dominant Asset Theorem} tells us to invest all available funds on the $j$th asset. In the sequel, the inequality~(\ref{eq: dominant ratio}) is called the \textit{dominant asset condition}.
	 As seen later in Section~\ref{SECTION: Dominant Ratio Trading Algorithm}, this condition allows us to construct a simple algorithm which may be useful for practical stock trading. 
	 
	 In the rest of this section, some other new optimality results are provided.
	
	\begin{lemma}[Expected Ratio Optimality]\label{lemma: Expected Ratio Optimality}
Let $K^*$ be the frequency-based optimal Kelly feedback gain. Then 
\[
\mathbb{E}\bigg[\frac{ 1+{K}^T\mathcal{X}_n}{1+{K^*}^T\mathcal{X}_n} \bigg] \leq  1
\]
for any $K$. 
In addition, we have
\[
\mathbb{E}\bigg[\log \frac{1+K^T\mathcal{X}_n}{1+{K^*}^T\mathcal{X}_n}\bigg] \leq 0
\]for any $K$.
	\end{lemma}

\begin{proof} Let $K$ be given.
From Theorem~\ref{thm: KKT optimality}, it follows that for a Kelly optimal feedback gain $K^*$, we have
\[
\mathbb{E}\bigg[ \frac{1+X_{n,i}}{1+{K^*}^T\mathcal{X}_n} \bigg] \leq 1
\] for all $i=1,\dots, m$. 
Multiplying this inequality by $K_i$ and summing over~$i$, we obtain
\[
\sum_{i=1}^m K_i \mathbb{E}\bigg[ \frac{1+X_{n,i}}{1+{K^*}^T\mathcal{X}_n} \bigg] \leq \sum_{i=1}^m K_i = 1
\]
which is equivalent to
\[
 \mathbb{E}\bigg[\frac{ 1+{K}^T\mathcal{X}_n}{1+{K^*}^T\mathcal{X}_n} \bigg] \leq  1.
\]
To complete the proof, we invoke Jensen's inequality on the quantity $\mathbb{E}\left[\log \frac{1+K^T\mathcal{X}_n}{1+{K^*}^T\mathcal{X}_n}\right]$ and observe that
\begin{align*}
\mathbb{E}\bigg[\log \frac{1+K^T\mathcal{X}_n}{1+{K^*}^T\mathcal{X}_n}\bigg] &\leq  \log \mathbb{E}\bigg[ \frac{1+K^T\mathcal{X}_n}{1+{K^*}^T\mathcal{X}_n}\bigg] \\[.5ex]
& \leq \log 1 = 0.
\end{align*}
Hence, the proof is complete.
\end{proof}

\textbf{Remark:} Lemma~\ref{lemma: Expected Ratio Optimality} above tell us that the frequency-based Kelly optimal portfolio also maximizes the expected relative wealth $\mathbb{E}[\frac{1+K^T\mathcal{X}_n}{1+{K^*}^T\mathcal{X}_n}]$. In addition, we note that the for any $K$,
\begin{align*}
1+K^T\mathcal{X}_n &= 1+\sum_{i=1}^{m}K_i\mathcal{X}_{n,i} \\
& \geq 1+\min_j\mathcal{X}_{\min,j}\sum_{i=1}^{m}K_i\\
& >0.
\end{align*}
Hence, the ratio $\frac{1+K^T\mathcal{X}_n}{1+{K^*}^T\mathcal{X}_n} > 0$. Now using the Markov inequality, the condition  \[
\mathbb{E}\bigg[\frac{ 1+K^T\mathcal{X}_n }{1+{K^*}^T\mathcal{X}_n}\bigg] \leq 1
\]for any $K$ implies that
\[
P \left(\frac{ 1+K^T\mathcal{X}_n }{1+{K^*}^T\mathcal{X}_n} > c\right) \leq \frac{1}{c}
\] for any $c >0.$ The following lemma indicates a stronger result on the asymptotic relative optimality of $K^*$. 

\begin{lemma}[Asymptotic Relative Optimality]\label{lemma: Asymptotic Relative Optimality}
The optimal feedback vector $K^*$ is such that
\[
\limsup_{n \to \infty} \frac{1}{n}\log \frac{1+K^T\mathcal{X}_n}{1+{K^*}^T\mathcal{X}_n}\leq 0
\] with probability one.
\end{lemma}

\begin{proof}
The idea of the proof is very similar to the one presented in \cite[Theorem~16.3.1]{Cover_Thomas_2012}. However, for the sake of completeness, we provide our own proof here.	Recalling Lemma~\ref{lemma: Expected Ratio Optimality}, we have
\[
\mathbb{E}\bigg[\frac{1+K^T\mathcal{X}_n}{1+{K^*}^T\mathcal{X}_n}\bigg] \leq 1
\]
and Markov inequality tell us that
\[
P \left(\frac{1+K^T\mathcal{X}_n}{1+{K^*}^T\mathcal{X}_n} > c_n\right) \leq \frac{1}{c_n}
\] for any $c_n >0.$ Hence,
\[
P\left(  \frac{1}{n} \log \frac{1+K^T\mathcal{X}_n }{1+{K^*}^T\mathcal{X}_n }>\frac{1}{n} \log c_n \right) \leq \frac{1}{c_n}.
\]
Take $c_n := n^2$ and summing all $n$, we have
\[
\sum_{n=1}^\infty P \left( \frac{1}{n} \log \frac{1+K^T\mathcal{X}_n }{1+{K^*}^T\mathcal{X}_n }>\frac{2\log n}{n}  \right) \leq \sum_{n=1}^\infty \frac{1}{n^2} < \infty.
\]
Therefore, applying the Borel-Cantelli Lemma; e.g., see \cite{Rosenthal_probability}, it leads to
\[
P \left(\frac{1}{n} \log \frac{1+K^T\mathcal{X}_n }{1+{K^*}^T\mathcal{X}_n }>\frac{2\log n}{n} \,\, \text{infinitely often} \right) = 0.
\]
Thus, there exists $N>0$ such that for all $n\geq N$, we have
\[
\frac{1}{n} \log \frac{1+K^T\mathcal{X}_n }{1+{K^*}^T\mathcal{X}_n }\leq \frac{2\log n}{n}.
\]
It follows that
\[
\limsup_{n \to \infty} \frac{1}{n}\log \frac{1+K^T\mathcal{X}_n }{1+{K^*}^T\mathcal{X}_n }\leq 0
\]
with probability one.
\end{proof}

\vspace{3mm}
\textbf{Remark:} Note that for $n\geq 1$, $V(n) = (1+K^T \mathcal{X}_n)V(0)$, thus, Lemma~\ref{lemma: Asymptotic Relative Optimality} implies that
\[
\limsup_{n \to \infty} \frac{1}{n} \log \frac{V(n)}{V^*(n)} \leq 0
\]
with probability one
where $V^*(n) = (1+{K^*}^T \mathcal{X}_n)V(0)$.

%
%

\section{Dominant Ratio Trading Algorithm}\label{SECTION: Dominant Ratio Trading Algorithm}
Besides the theoretical interests, as mentioned in Section~\ref{SECTION: Main Results}, we view that Theorem~\ref{thm: KKT optimality} and Extended Dominate Asset Theorem~\ref{thm: Extended Dominant Asset Theorem} may be useful to design an algorithm for practical stock trading. 
The main idea is to take advantage of the Dominant Asset Condition stated in Theorem~\ref{thm: Extended Dominant Asset Theorem}; i.e.,
\[
\mathbb{E}\left[ \frac{1+X_i(k)}{1+X_j(k)}\right] \leq 1,
\] 
if it holds, then we set $K_j^* = 1$;
otherwise, $K_j^* = 0.$

\subsection{Bridging Theory and Practice}
 To implement the idea described above, we proceed as follows: Using $s_i(k)$ to denote the  $k$th daily \textit{realized prices} for the $i$th stock, we calculate the associated \textit{realized return}, call it $x_i(k)$, where
\[
x_i(k) := \frac{s_i(k+1)-s_i(k)}{s_i(k)}
\] for $i=1,2,\dots,m$. 
It should be noted that, in practice, the realized returns $x_i(k)$ are often nonstationary. Hence, when testing the dominant asset condition, we work with a sliding window consisting of the most recent~$M$ trading steps.\footnote{Again, we note here that the unit of ``steps" here can be any time stamp such as milliseconds, minutes, days, months, etc.}
That is, we estimate the expected ratio in the Dominant Asset Condition by 
\[
R_{ij}(k):= \frac{1}{M} \sum_{\ell=0}^{M-1}\frac{1+x_i(k-\ell)}{1+x_{j}(k-\ell)}.
\]
Then, if $R_{ij} \leq 1$ for all $i \neq j$, we set $K_j^*(k)=1$; otherwise, we set $K_j^*(k)=0.$ We call the procedure above the Dominant Ratio Trading Algorithm.
An illustrative example using historical prices data is provided in the next subsection to follow.

%

%

\subsection{Illustrative Example Via Back-Testing}
Consider a one-year long portfolio consisting of three assets with duration from February~14, 2019 to February~14,~2020: Vanguard Total World Stock Index Fund ETF Shares \mbox{(Ticker: VT)}, Vanguard Total Bond Market Index Fund ETF Shares (Ticker: BND), and Vanguard Total World Bond EFT \mbox{(Ticker: BNDX)} where the price trajectories are shown in Figure~\ref{fig:stockprices}.\footnote{The data are provided by Wharton Research Data Services. 
}  

\begin{figure}
	\centering
	\includegraphics[width=.9\linewidth]{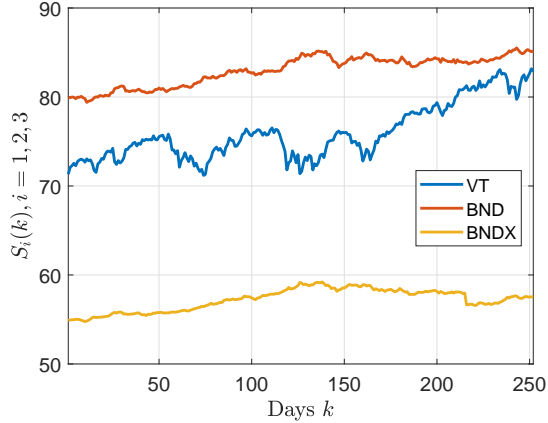}
	\caption{Daily Closing Stock Prices $S_i(k), i=1,2,3$ for VT, BND, and BNDX, respectively.}
	\label{fig:stockprices}
\end{figure}

Begin with initial account value $V(0)=\$1$, we implement the algorithm described above using a window size $M=20$ days. That is, the initial trade is triggered after receiving the first twenty daily prices data.
We ran MATLAB script and plot a typical trading performance in terms of the trajectory of account value~$V(k)$, which is shown in Figure~\ref{fig:goldenratiotest3assets}. 
In Figure~\ref{fig:goldenratiotest3assets}, we find that the account value obtained by Dominant Ratio Trading Algorithm is increasing from $V(0)=1$ to $V(252) \approx 1.23$, which yields a returns about $23\%$ and is obviously higher than the account value obtained by standard buy and hold strategy.
We also reported the corresponding trading signal~$K_i(k)$ for $i=1,2,3$ in Figure~\ref{fig:k1k2k3} where a flavor of bang-bang control is seen. 
To close this section, we also tested various sliding window sizes using equally-spaced $M=1,5,15,\dots,60$ with increment $5$ between elements and we seen that the algorithm produces similar trading performance to the one seen in Figure~\ref{fig:goldenratiotest3assets}.
This example provides a potential for bridging the theory and practice in stock trading.  Further developments along this line might be fruitful to pursue as a direction of future research. For example, an initial computational complexity analysis and trading with various stocks may be of the next interests to pursue.

\begin{figure}
	\centering
	\includegraphics[width=0.9\linewidth]{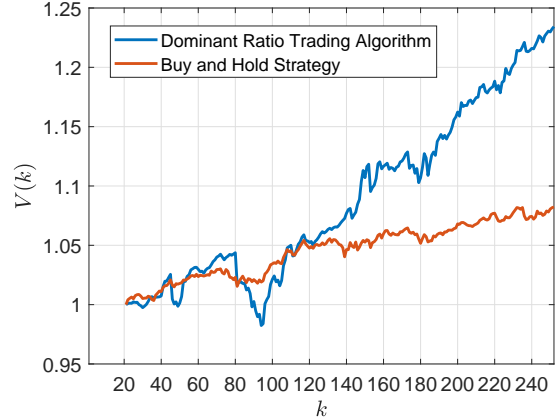}
	\caption{Trading Performance Comparison: Dominant Ratio Trading Algorithm with $M=20$ versus Buy and Hold strategy.}
	\label{fig:goldenratiotest3assets}
\end{figure}


\begin{figure}
	\centering
	\includegraphics[width=0.9\linewidth]{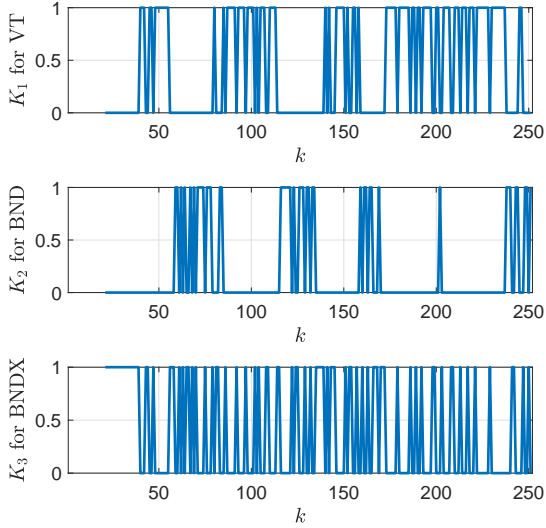}
	\caption{Feedback Gains $K_i(k)$ with $i=1,2,3$ for VT, BND, and BNDX, respectively. One sees a bang-bang flavored control signals.}
	\label{fig:k1k2k3}
\end{figure}


\section{Conclusion and Future Work}	\label{SECTION: Conclusion}
In this paper, we studied necessary and sufficient conditions for the frequency-based optimal Kelly portfolio. 
With the aid of these conditions, we derived various different optimality characterizations such as expected ratio optimality, asymptotic relative optimality, and Extended Dominant Asset Theorem.
Moreover, to bridge the theory and practice, we used the notion of dominant asset to construct a trading algorithm which indicates the trader when to invest all available funds into the dominant asset. 

Regarding further research, one obvious continuation would be to study the case when $K_i<0$ is allowed; i.e., short selling should be considered as a next level extension of the formulation.  
In this situation, we envision a similar results along the lines of those given here.  
In addition, it would be of interest to relax some of the assumptions in the formulation from i.i.d. return sequences to time-dependent sequences. 

Finally, 
for cases when the distribution model for returns~$X_i(k)$ is either partially known or completely unknown, it would be of interest to study the extent to which the theory in this paper can be extended. For example, the line along the data-driven algorithm described in Section~\ref{SECTION: Dominant Ratio Trading Algorithm} might be~helpful. 

\section{ACKNOWLEDGMENTS}	
The author thanks Professor B. Ross Barmish and Professor John A. Gubner for leading the author into this field and seeing the potentiality and applicability of control theory. 

\appendices

\section{Survival Considerations}
\label{SECTION: Surivial}
\vspace{0mm}
%
In the context of stock trading, the very first goal for a trader is to assure that the bankruptcy would never occur for the entire trading period; i.e., one must assure $V(k)>0$ for all~$k$. 
If this is the case, we say the trades are \textit{survival}.\footnote{
As stability is to the classical control system, so is survivability to the financial system.
 In fact, in our prior work \cite{Hsieh_Barmish_Gubner_2019_TAC}, the survivability problem is regarded  as a state positivity problem.} Below, we provide a result which indicates that the any feedback gain $K$ satisfying the constraint set $\mathcal{ K}$ considered in Section~\ref{SECTION: Problem Formulation} assures survival.

\vspace{5mm}\noindent
\textbf{Lemma:}\label{lemma: Survival for V(k+1)}
\textit{If $K \in \mathcal{ K}$, then
	$V(n)>0$ for all $n\geq 1$.}

\begin{proof} 

We first note that for $n=1$, the account value is
\begin{align*}
	V(1) = (1+K^T\mathcal{X}_1)V(0) 
	= (1+K^TX(k))V(0)>0.
\end{align*}
	Now, to show \mbox{$V(n) >  0$} for $n>1$, we observe that 
	\begin{align*}
	V(n) &= (1+K^T\mathcal{X}_n)V(0)\\[.5ex]
	&= \left(1+ \sum_{i=1}^m K_i \left(\prod_{k=0}^{n-1}(1+X_i(k))-1\right) \right)V(0) \\[.5ex]
	& \geq  \left(1+ \sum_{i=1}^m K_i \mathcal{X}_{i,\min} \right)V(0)
	\end{align*}
	where $\mathcal{X}_{i,\min} := (1+X_{\min,i})^n-1 > -1$ for all $i,$ Hence,
	\begin{align*}
	V(n) & \geq \left(1+\min_{i=1,\dots,m}\mathcal{X}_{i,\min}  \sum_{j=1}^m K_j \right)V(0)\\[.5ex]
	& = \left(1+\min_{i=1,\dots,m}\mathcal{X}_{i,\min}  \right)V(0)\\[.5ex]
	&>0
	\end{align*}
where the last inequality holds since $\mathcal{X}_{i,\min}>-1$ for all $i$ implies $\min_i \mathcal{X}_{i,\min}>-1$
	and the proof is complete.
\end{proof}	

	

\end{document}